\documentclass[12pt]{article}

\usepackage[utf8]{inputenc}
\usepackage[T1]{fontenc}
\usepackage[english]{babel}

\usepackage[a4paper,vmargin={2.5cm,2.5cm},hmargin={2cm,2cm}]{geometry}
\usepackage{mathpazo}
\usepackage[font=sf, labelfont={sf,bf}, margin=1cm]{caption}

\usepackage[pdftex]{hyperref}
\usepackage[expansion]{microtype}
\usepackage[pdftex]{color,graphicx}

\usepackage[fleqn]{amsmath}
\usepackage{amsfonts,amssymb,amsthm,mathrsfs}
\usepackage{stackrel}

\theoremstyle{plain}
\newtheorem{theorem}{Theorem}
\newtheorem{corollary}{Corollary}
\newtheorem{proposition}[corollary]{Proposition}

\newtheorem{definition}[corollary]{Definition}

\theoremstyle{remark}

\newcommand{\CC}{\mathrm{Cat}}
\newcommand{\FM}{\mathrm{FM}}
\newcommand{\HH}{\mathcal{H}}
\newcommand{\NN}{\mathbb{N}}
\newcommand{\RR}{\mathbb{R}}
\newcommand{\ER}{Erd\H{o}s-Rényi }

\def\presuper#1#2%
  {\mathop{}%
   \mathopen{\vphantom{#2}}^{#1}%
   \kern-\scriptspace%
   #2}

\hypersetup{
    pdfauthor   = {N. Enriquez, L. Ménard},%
    pdftitle    = {Spectra of large diluted but bushy random graphs}}

\begin{document}

\title{\bf Spectra of large diluted but bushy random graphs}
\author{Nathanaël Enriquez, Laurent Ménard}
\date{}
\maketitle


\begin{abstract}
We compute an asymptotic expansion in $1/c$ of the limit in $n$ of the empirical spectral measure of the adjacency matrix of an Erd\H{o}s-Rényi random graph with $n$ vertices and parameter $c/n$.  We present two different methods, one of which is valid for the more general setting of locally tree-like graphs.
The second order in the expansion gives some information about the edge of the
spectrum.
\end{abstract}


\section{Introduction}

It is a consequence of the celebrated result of Wigner \cite{Wigner} that the limit of the empirical spectral measure of the adjacency matrix of a large Erd\H{o}s-Rényi random graph with fixed parameter $p$ is the semi-circle law. This fact remains valid when $p$ is allowed to depend on the size $n$ of the graph, as long as $np \to \infty$.

In the diluted regime, i.e. when  $np$ converges to a constant $c$, it has been proved that the empirical spectral distribution still converges, when properly rescaled, to a probability distribution $\mu^c$ (see the work of Zakharevich \cite{Z06}). However, this measure is far from being well understood. Let us mention two recent  breakthroughs: in \cite{BLS}, Bordenave, Lelarge and Salez computed the mass $\mu^c(\{ 0 \})$, and more recently, Bordenave Sen and Virag proved in \cite{BSV} that $\mu^c$ has a continuous part if and only if $c > 1$.

In the present paper, we focus on the study of $\mu^c$ for large $c$ and
describe how $\mu^c$ differs from the semi-circle law. More precisely, we
compute an asymptotic expansion in $1/c$ of $\mu^c$ (see Theorem \ref{ER} and
\ref{edge} for precise statements). The second order in the expansion gives some
information about the edge of the spectrum (see Section \ref{sec:order2}).

\bigskip

In a famous paper \cite{BS}, Benjamini and Schramm introduced the so-called notion of local convergence for sequences of graphs. In this terminology, \ER graphs of parameter $c/n$ converge locally towards a Galton-Watson tree with Poisson offspring distribution. More generally,
the configuration model introduced by Bollob{\'a}s \cite{Bollobas} gives a
generic construction of random graphs converging locally towards random trees.

Bordenave and Lelarge proved in \cite{BL10} that, for random graphs converging locally to random trees, the expectation of the spectral measure of the limiting tree is the limit of the spectral measures of the random graphs. This allows us to extend the computations we did for \ER random graphs to any sequence of growing random graphs converging locally to a random tree.

\bigskip

Our method is based on the computation of the moments of $\mu^c$ and the underlying combinatorics. An interesting aspect of the present work is that this combinatorics takes a very different form depending on whether the computations are made directly on the whole finite graph --- as we do in the special case of \ER random graphs --- or on their local limit.

On the basis of former works of Khorunzhyi, Shcherbina and Vengerovsky \cite{KSV04}, and also Bordenave, Lelarge and Salez \cite{BL10,BLS}, it is natural to try to use the resolvent method for our computations. We show in the appendix that this method faces some serious problems since it involves, during intermediate computations, quantities that are strongly diverging.

\section{Spectral measure of the Erd\H{o}s-Rényi random graph} \label{sec:ER}

Let $X_n$ be the adjacency matrix of the \ER random graph $G(n,c/n)$. It is a symmetric $n \times n$ random matrix having a null diagonal and whose entries above the diagonal are i.i.d. Bernoulli random variables with parameter $c/n$. We define the normalised spectral measure of $X_n$ by
\[
\mu_n^{c} = \frac{1}{n} \sum_{\lambda \in \textrm{Sp} \left( c^{-1/2} X_n \right)} \delta_{\lambda}.
\]
As in the Gaussian Unitary Ensemble, we rescale $X_n$ by $c^{-1/2}$ so that the variance of off diagonal entries is asymptotically equal to $1/n$ . Indeed, if $i \neq j$:
\[
\mathrm{E}  [ X_n ^2] = \frac{c}{n}.
\]
It is of common knowledge (see e.g. \cite{KSV04,BL10,Z06}) that
\begin{itemize}
\item the sequence $\left( \mu_n^{c} \right)_{n \geq}$ converges weakly to a probability measure $\mu^c$ as $n \to \infty$.
\item when $c \to \infty$, the measure $\mu^c$ converges to Wigner's semi-circle law.
\end{itemize}
Our first result gives an asymptotic expansion of $\mu^c$ as $c \to \infty$:
\begin{theorem}
\label{ER}
For a measure $\mu$, denote $m_k(\mu)$ the moment of order $k$ of $\mu$ when it exists  (i.e. when $\int |x|^k
| d \mu (x) |
 < \infty $).
 One has, for every $k \geq 0$ and as $c \to \infty$
\[
m_k(\mu^c) = m_k \left( \sigma + \frac 1 c \sigma^{\{1\}} \right) + o \left( \frac 1 c \right)
\]
where $\sigma$ is the semi-circle law having density
\[ 
\frac{1}{2\pi} \sqrt{4-x^2} \mathbf{1}_{|x|<2}
\] 
and $\sigma^{\{1\}}$ is a measure with total mass $0$ and density
\[
\frac{1}{2 \pi} \frac{x^4 -4x^2 +2}{\sqrt{4-x^2}} \,
\mathbf{1}_{|x|<2}.
\]
\end{theorem}

\begin{proof}
Let us compute the moments of $\mu_n^{c}$. Let $l$ be an integer, 
\begin{align}
m_l \left(\mu_n^{c} \right) & =  E \left[ \int x^l d \mu_n^{c} \right] = \frac{1}{n} E \left[ \textrm{Tr} \left( \frac{1}{\sqrt{c}} X_n \right)^l \right] = \frac{1}{n c^{l/2}} \sum_{1 \leq j_1, \ldots , j_l \leq n} E \left[ X_n (j_1,j_2) \cdots X_n(j_l,j_1) \right]. \label{deftrace}
\end{align}

First, let us prove that odd moments converge to $0$ as $n$ goes to infinity. For this purpose, we notice the useful following fact : in the sum \eqref{deftrace}, the contribution of all the sequences $j_1, \ldots , j_l$ where a pair $\{m,m'\} $ appears an odd number of times among the pairs of the form $\{j_i , j_{i+1} \}$ goes to $0$ as $n$ goes to infinity.

Indeed, since $E [ X_n(i,j)^{l} ] = \frac{c}{n}$ for all $l,i$ and $j$, if we fix a sequence $j_1, \ldots , j_l$, then 
$$
E \left[ X_n (j_1,j_2) \cdots X_n(j_l,j_1) \right] = \left( \frac{c}{n} \right) ^{a(j_1, \ldots , j_l)}
$$
where $a(j_1, \ldots , j_l)$ denotes the number of different pairs of the form  $\{j_i , j_{i+1} \}$.
Let $s(j_1, \ldots , j_l)$ be the number of distinct integers in the sequence $j_1, \ldots , j_l$. The contribution in \eqref{deftrace} of all sequences $j_1, \ldots , j_l$ such that $a(j_1, \ldots , j_l) = \alpha$ and such that $s(j_1, \ldots , j_l) = \zeta$ is then smaller than $\frac{1}{n c ^{l/2}} \left( \frac{c}{n} \right)^{\alpha} n^\zeta$. A non null asymptotic contribution arise only for $\zeta = \alpha + 1$.

Note now that a sequence $j_1, \ldots , j_l$ defines a connected graph whose vertex set is $\{j_1, \ldots , j_l\}$ and whose edges are the pairs $\{j_i , j_{i+1} \}$. This graph has $s(j_1, \ldots , j_l)$ vertices and $a(j_1, \ldots , j_l)$ edges and is therefore a tree when $s=a+1$.

The sequence $j_1, \ldots , j_l, j_1$ is then a closed path of length $l$ on this tree and must therefore be of even length.

\bigskip

From now on, we consider even moments so that $l = 2k$. Let us take a closer look at the  sequences $j_1, \ldots , j_{2k}$ such that $a(j_1, \ldots , j_{2k}) = \alpha$ and $s(j_1, \ldots , j_{2k}) = \alpha + 1$ for fixed $\alpha$. When $n$ goes to infinity, their contribution is equal to $c^{\alpha - k }$ multiplied by the number of closed paths of length $2k$ on trees with $\alpha$ edges, with the constraints that a path starts and ends at the root of the tree and visits every vertex.

Based on this observation, we can study the asymptotic expansion for large $c$ of the asymptotic (in $n$) spectral measure of $X_n$ via its moments. The mean value in this expansion comes from the special case $\alpha = k$. Each sequence $j_1, \ldots , j_{2k},j_1$ such that $a(j_1, \ldots , j_{2k}) = k$, $s(j_1, \ldots , j_{2k}) = k + 1$  is then the countour function of a tree with $k$ edges. The total contribution of these sequences is therefore the $k$-th Catalan number $\CC(k)$. This explains that when $c$ goes to infinity, the asymptotic (in $n$) spectral measure of the random graph is close to Wigner's semi-circle law.

The next term in the asymptotic expansion is of order $1/c$ and comes from the case $\alpha = k -1$. Whereas in the previous case each edge of the tree was visited exactly twice, in this case, exactly one edge is visited four times and all the other edges are visited twice. Therefore we have to enumerate sequences of the form 
$$\mathcal{S} = j_1 \mathbf{S}_1 i \,  j \mathbf{S}_2  j  \, i \mathbf{S}_3 i \, j \mathbf{S}_4 j \, i \mathbf{S}_5 j_1 $$
such that
\begin{itemize}
\item the sequences $\mathbf{S}_1, \ldots ,\mathbf{S}_5$ have no common term and do not contain any of the three integers $i$ or $j$. In addition, the sequences $\mathbf{S}_2$, $\mathbf{S}_3$ and $\mathbf{S}_4$ do not contain $j_1$
\item the sequences $j_1 \mathbf{S}_1 i \mathbf{S}_5 j_1$, $j \mathbf{S}_2 j$, $i \mathbf{S}_3 i$ and $j \mathbf{S}_4 j$ are the contour functions of trees with respectively $p_1$, $p_2$, $p_3$ and $p_4$ edges satisfying $p_1 + p_2 + p_3 + p_4 = k -2$;
\end{itemize}
up to the specific values of the integers appearing in the whole sequence.

The sequence $j_1 \mathbf{S}_1 i \mathbf{S}_5 j_1$ corresponds to a rooted tree
with $p_1$ edges and a marked corner (adjacent to the vertex $i$) where the
trees corresponding to the three other sequences are inserted (see Figure
\ref{ERorder1_fig} for an illustration). Note that when $p_1 = 0$, the sequence
$\mathcal{S}$ boils down to $i \, j \mathbf{S}_2 j  \, i \mathbf{S}_3 i \, j
\mathbf{S}_4 j \, i$.

\begin{figure}[t!]
\begin{center}
\includegraphics[width=5cm]{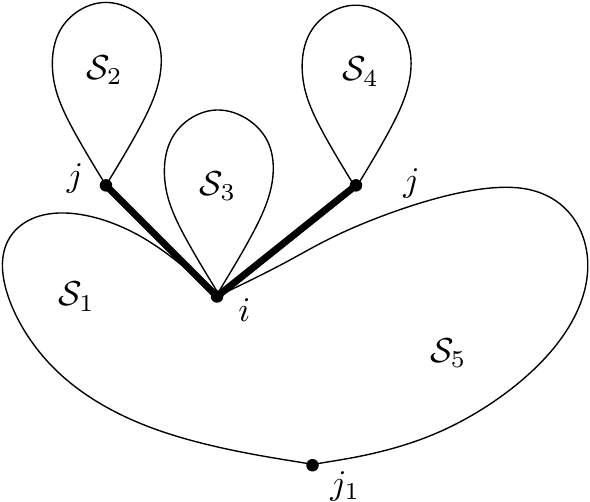}
\caption{\label{ERorder1_fig} Labeled tree with contour $\mathcal{S}$.}
\end{center}
\end{figure}

There are $(2p_1 +1) \CC(p_1)$ such trees with a marked corner. The term of order $1/c$ in the asymptotic expansion of the moment $m_{2k}$ is then:
\[
m_{2k}^{\{1\}} = \sum_{p_1 +\cdots + p_4 = k-2} (2p_1 +1) \CC(p_1) \cdots \CC(p_4).
\]

The generating function $S^{\{1\}}$ of the sequence $\left( m_{2k}^{\{1\}} \right)_{k \geq 0}$ is given by
\begin{align*}
S^{\{1\}} (x) & = \sum_{k \geq 0} m_{2k}^{\{1\}} x^k  = \sum_{k \geq 0} x^2 \sum_{p_1 +\cdots + p_4 = k-2}  (2p_1 +1) \CC(p_1) x^{p_1} \cdots \CC(p_4) x^{p_4} \\
& = x^2 \left( 2x T'(x) T(x)^3 + T(x)^4\right)
\end{align*}
where $T$ is the generating function of Catalan numbers:
$T(x) = \frac{1 - \sqrt{1-4x} } {2x}$ for $|x| < 1/4$.

We show now that the sequence $\left( m_{l}^{\{0\}} + \frac1c m_{l}^{\{1\}} \right)_{l \geq 0}$ is the sequence of the moments of the probability measure announced in the theorem. A formal proof would consist in the computation of the moments of this measure, but we will rather show how to compute this density from the moments via the Stieljes transform. Define
\begin{align*}
\HH^{\{1\}} (z) & = \frac1z S^{\{1\}} \left( \frac{1}{z^2}\right).
\end{align*}
One has
\begin{align*}
\HH^{\{1\}} (z)
= \frac{2}{z^7} T' \left(\frac{1}{z^2} \right) T \left( \frac{1}{z^2} \right)^3 + \frac{1}{z^5} T \left( \frac{1}{z^2} \right)^4 = - \HH'(z) \HH(z)^3
\end{align*}
where
$\HH(z) = \frac1z T \left( \frac{1}{z^2} \right) = \frac12 \left( z - \sqrt{z^2 -4} \right)$
is the Stieljes transform of the semi-circle law $\sigma$ . Therefore
\begin{equation}
\label{H1}
\HH^{\{1\}}(z) = \frac{1}{16} \frac{\left( z -
\sqrt{z^2-4}\right)^4}{\sqrt{z^2-4}}.
\end{equation}
This corresponds to the Stieljes transform of a signed measure $\sigma^{\{1\}}$ with
null total mass and having density
\[
 \lim_{\varepsilon \to 0} - \frac1\pi \, \mathrm{Im}\left( \HH^{\{1\}}(x+i
\varepsilon) \right) = \frac{1}{2 \pi} \frac{x^4 -4x^2 +2}{\sqrt{4-x^2}} \,
\mathbf{1}_{|x|<2}.
\]
\end{proof}

\section{Spectral measure of Unimodular Galton Watson trees}

\subsection{Asymptotic expansion of the spectral measure}

\ER random graphs with parameter $c/n$ are known to converge locally
towards Galton Watson trees with a Poisson reproduction law. The aim of this
section is to extend our computation to the so-called "locally tree-like
graphs", whose local limit are Unimodular Galton Watson trees \cite{AL07}. The
classical setting providing such limits is the configuration model we now
present (see \cite{Bobook} Section 2.4 for details).

Let $\mathbf{p}(c) = (p_k(c))_{k \geq 0}$ be a probability measure on $\NN$ with finite mean $c$. We can construct a random graph with $n$ vertices associated to $\mathbf{p}(c)$ by the following procedure.
In a first step, we choose a sequence $\left( d_i \right)_{i=1\ldots n}$ of i.i.d. random variables with distribution $\mathbf{p}(c)$.
In a second step, for every $i$ between $1$ and $n$, we make $d_i$ half-edges start from vertex $i$. Assuming the sum of the $d_i$'s is even we can connect the half edges by pairs and obtain a graph (possibly with loops and multiple edges). If the sum  of the $d_i$'s is not even, we increase the degree $d_n$ by $1$ as it will not change the local limits of the graphs.
In a third step, we choose a graph uniformly at random among all the graphs
obtained by connecting the half edges by pairs. In order to obtain a simple
graph we erase self-loops and merge multiple edges. We denote by
$G_n(\mathbf{p}(c))$ the random graph obtained by this device.

Viewed from a uniformly chosen vertex $\rho$, these graphs are
known to converge locally as their number of vertices grows to infinity towards
the Unimodular Galton Watson tree $UGW(\mathbf p (c))$ defined as follows.
The root of $UGW(\mathbf p (c))$ has a random number of children distributed according to $\mathbf{p}(c)$. Other vertices have independent numbers of children distributed according to the size biased version of $\mathbf{p}(c)$, namely the probability $\mathbf q (c)$ with weight sequence $q_k(c) = (k+1) p_{k+1}(c) / c $.

\bigskip

As we show in the following (see \cite{BL10} for details), the local  limit
$UGW(\mathbf p (c))$ contains all the material to identify the limiting spectral
measure of the initial graphs $G_n(\mathbf p (c))$. 
Denote the adjacency matrix of $G_n(\mathbf p (c))$ by $A_n(\mathbf p (c))$ and
consider the normalised spectral measures of $G_n(\mathbf p (c))$:
\[
\mu_n (\mathbf p (c))= \frac{1}{n} \sum_{\lambda \in \textrm{Sp} \left( c^{-1/2} A_n (\mathbf p (c))\right)} \delta_{\lambda}.
\]
One has
\begin{align}
& E  \left[ m_k (\mu_n(\mathbf p (c))) \right] \notag \\
& =  E \left[ \int x^k d
\mu_n(\mathbf p (c)) \right] \notag \\
& = \frac{1}{n} E \left[ \textrm{Tr} \left( \frac{1}{\sqrt{c}} A_n (\mathbf p (c))\right)^k \right] \notag \\
& = \frac{1}{c^{k/2}} E \left[ \frac1n \times \text{number of loops of length $k$ in $G_n(\mathbf p (c))$} \right] \notag \\
& = \frac{1}{c^{k/2}} E \left[ \text{number of loops started at $\rho$ and of length $k$ in $G_n(\mathbf p (c))$} \right] \notag \\
& \underset{n \to \infty}{\longrightarrow} \frac{1}{c^{k/2}} E \left[ \text{number of loops started at the root and of length $k$ in $UGW(\mathbf p (c))$} \right].
\label{CVloc}
\end{align}
If $T$ is a random rooted tree (finite or infinite), a loop inside $T$ is of
even length, therefore we define, for every $k \geq 0$,
\begin{equation*}
L_k(T) = E \left[ \text{number of loops started at the root and of length $2k$ in $T$} \right]
\end{equation*}
so that \eqref{CVloc} can be written as
\begin{align}
E & \left[ m_{2k +1}  (\mu_n(\mathbf p (c))) \right] 
 \underset{n \to \infty}{\longrightarrow} 0 \notag \\
E & \left[ m_{2k} (\mu_n(\mathbf p (c))) \right] 
\underset{n \to \infty}{\longrightarrow} \frac{1}{c^{k}} L_k \left( UGW ( \mathbf p (c) ) \right).
\label{CVloc2}
\end{align}
Up to the factor $c^{-k}$, the right hand side of \eqref{CVloc2} is the $2k$-th
moment of the spectral measure $\mu \left( UGW (\mathbf p (c)) \right)$ of
$UGW(\mathbf p (c))$ defined in \cite{BL10} (this paper also states that
$\mathbf p (c)$ must have a finite variance for this measure to be defined and
characterised by its moments, which will be the case in the following). As in
the previous section, we are interested in the asymptotic expansion of this
measure as $c \to \infty$. In the sequel, we
assume that the distribution $\mathbf{p} (c)$ is \emph{concentrated} around its
mean $c$. 
More precisely, we make the following assumption on the
factorial moments of $\mathbf p (c)$: there exists $\alpha > 0$ and a function $f : \NN
\to \RR$  such that for every $k \in \NN$ one has, when $c \to \infty$
\begin{equation}
\frac{\FM_k(\mathbf p (c))}{c^k} := \frac{E_{\mathbf{p}(c)}\left[X(X-1) \ldots
(X-k+1) \right]}{c^k}  = 1 + \frac{f(k)}{c^{\alpha}} + o \left( \frac 1
{c^{\alpha}} \right).
\label{DLcum}
\end{equation}
Note that this condition implies that $\mathbf p (c)$ has finite moments of all orders and that $f(0) = f(1) = 0$.

\begin{theorem}
\label{UGW}
Let $\left( \mathbf{p}(c) \right)_{c \geqslant 0} = \left( (p_k(c))_{k \geq 0}
\right)_{c \geqslant 0}$ be a family of probability measures on $\NN$ with
finite mean $c$ satisfying \eqref{DLcum}. One has, for every $k \geq 0$ and as
$c \to \infty$
\[
m_k(\mu(\mathbf p (c)) ) = m_k \left(\sigma + \frac 1 {c^{\alpha}}  \sigma_f^{\{1\}}  + \frac 1 c \sigma^{\{1\}}  \right)  + o \left( \frac 1 {c^{1 \wedge \alpha}} \right)
\]
where $\sigma$ and $\sigma^{\{ 1\}}$ are defined in Theorem \ref{ER} and $\sigma_f^{\{1\}}$ is a measure of total mass $0$ whose moments generating function is given by
\[
\frac{2 F(x T(x))}{ 1 - x T(x)^2}
\]
where $T$ is the generating series of Catalan numbers and $F(x) = \sum_{k \geq
0} f(k) x^k$.
\end{theorem}

\begin{proof}
Let us focus on the right hand side of \eqref{CVloc2}.  As with \ER
random graphs, the key to obtain an asymptotic expansion of $c^{-k} L_{k} \left(
UGW ( \mathbf p (c) \right)$ when $c \to \infty$ resides in the fact that the
contribution of loops with repeated edges is of smaller order than the
contribution of loops with no repeated edges. The notion of loops with repeated
edges will be instrumental in the following, therefore we introduce the 
following notation.
\begin{definition}
Given a rooted tree,  we denote
\begin{itemize}
\item by $0$-loops the loops started at the root and visiting each edge of the
tree either twice (a first time from the root and a second time towards the
root) or not at all;
\item by $1$-loops the loops started at the root and visiting each edge of the
tree either twice (a first time from the root and a second time towards the
root) or not at all with the exception of one edge visited four times.
\end{itemize}
Furthermore, if $T$ is a random tree, we denote by $L_{k}^{(0)} (T)$ the expected number of $0$-loops in $T$ of length $2k$ and $L_{k}^{(1)} (T)$ the expected number of $1$-loops in $T$ of length $2k$
\end{definition}

\begin{figure}[t!]
\begin{center}
\includegraphics[width=12cm]{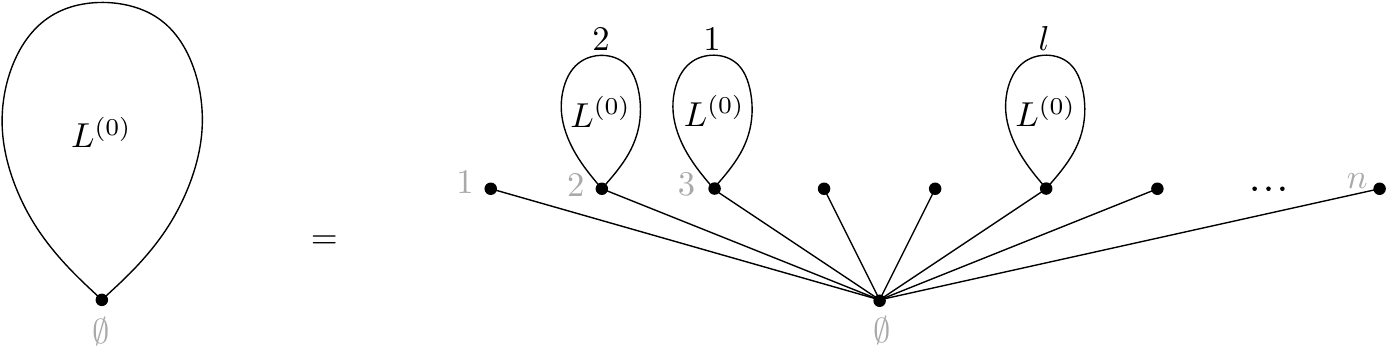}
\caption[0-loops]{\label{fig:0loops} Illustration of \eqref{Lk0} satisfied by
$L_{k}^{(0)} \left( UGW ( \mathbf p (c)) \right)$.
Numbers in gray next to vertices
count the children of the root, while numbers in black give the order of
appearance of first generation vertices in the loop.
}

\end{center}
\end{figure}

Let us start by studying $ L_{k}^{(0)} \left( UGW ( \mathbf p (c)) \right)$. A
$0$-loop can be decomposed into a sequence of visits of $l \in \{1, \ldots , k
\} $ distinct edges joining the root to a first generation vertex, each of these
visits being followed by a $0$-loop in the subtree of the descendants of the
associated first generation vertex. When $T$ is the tree $UGW \left( \mathbf p
(c)) \right) $, the root has $n$ children with probability $p_n (c) $ and the
subtrees of the descendants of these $n$ children are independant Galton Watson
trees with reproduction law $\mathbf q (c)$ denoted $GW \left( \mathbf q (c)
\right)$. See Figure \ref{fig:0loops} for an illustration. The choice of a
sequence of $l$ distinct first generation vertices among the $n$ children of the
root gives a factor $n (n-1) \ldots (n-l+1)$
This yields the following recursion equation for $k > 0$
\begin{align}
& c^{-k} L_{k}^{(0)} \left( UGW ( \mathbf p (c)) \right) \notag\\
& =
\sum_{l=1}^{k} c^{-l} \left( \sum_{n \geq l} n (n-1) \cdots (n- l +1) p_n(c) \right) \notag \\
& \qquad \qquad \times \sum_{k_1 + \cdots + k_l = k-l} c^{-k_1} L_{k_1}^{(0)} \left( GW ( \mathbf q (c)) \right) \times \cdots \times
c^{-k_l} L_{k_l}^{(0)} \left( GW ( \mathbf q (c)) \right) \notag \\
& = \sum_{l=1}^{k} \frac{\FM_l(\mathbf p (c))}{c^l} \sum_{k_1 + \cdots + k_l = k-l} c^{-k_1} L_{k_1}^{(0)} \left( GW ( \mathbf q (c)) \right) \times \cdots \times
c^{-k_l} L_{k_l}^{(0)} \left( GW ( \mathbf q (c)) \right).
\label{Lk0}
\end{align}
Note that in the above equation, the factors $c^{-k}$ play no combinatorial role and can be dropped resulting in a recursion relation for the $L_k$'s. However, we keep them in the formula since we have to deal with the moments of the spectral measure $c^{-k} L_k$. Similarly
\begin{align}
& c^{-k} L_{k}^{(0)} \left( GW ( \mathbf q (c) )\right) \notag \\
& =\sum_{l=1}^{k} \frac{\FM_{l}(\mathbf q (c))}{c^{l}} \sum_{k_1 + \cdots + k_l = k-l} c^{-k_1} L_{k_1}^{(0)} \left( GW ( \mathbf q (c)) \right) \times \cdots \times
c^{-k_l} L_{k_l}^{(0)} \left( GW ( \mathbf q (c)) \right) \notag \\
& = \sum_{l=1}^{k} \frac{\FM_{l+1}(\mathbf p (c))}{c^{l+1}} \sum_{k_1 + \cdots + k_l = k-l} c^{-k_1} L_{k_1}^{(0)} \left( GW ( \mathbf q (c)) \right) \times \cdots \times
c^{-k_l} L_{k_l}^{(0)} \left( GW ( \mathbf q (c)) \right). \label{Lk0GW}
\end{align}

Our aim in this section is to provide an asymptotic expansion in $c$ of the moments $c^{-k} L_{k} \left( UGW ( \mathbf p (c)) \right)$ with a precision of order $1 / {c^{1 \wedge \alpha}}$. A first step will be the asymptotic expansion of $c^{-k} L_{k}^{(0)} \left( UGW ( \mathbf p (c) )\right)$ derived later, but $1$-loops have a contribution of order $1/c$. Therefore, a recursion equation analogous to \eqref{Lk0} for $1$-loops is needed.

Denoting by $\presuper{2}{L}_{k}^{(0)} (T)$ the expectation of the number of
\emph{disjoint} and ordered pairs of $0$-loops of total length $2k$ in a random tree $T$, we have
\begin{align}
& c^{-k} L_{k}^{(1)} \left( UGW ( \mathbf p (c)) \right) \notag \\
& = \sum_{l=1}^{k-2} \frac{\FM_l(\mathbf p (c))}{c^l} \,
l \sum_{k' = 2}^{k-l} c^{-k'} L_{k'}^{(1)} \left( GW ( \mathbf q (c) ) \right)
\sum_{\substack{k_1 + \cdots + k_{l-1} \\ = k - k' - l}} \prod_{j=1}^{l-1} c^{-k_j}  
L_{k_j}^{(0)} \left( GW ( \mathbf q (c) ) \right) \notag \\
& + \sum_{l=1}^{k-1} \frac{\FM_l(\mathbf p (c))}{c^{l+1}}
\frac{l(l+1)}{2} \sum_{k' = 0}^{k-(l+1)} c^{-k'} \, \presuper{2}{L}_{k'}^{(0)}
\left( GW ( \mathbf q (c)) \right)
\sum_{\substack{k_1 + \cdots + k_{l-1} \\ = k - k' - l -1} } \prod_{j=1}^{l-1} c^{-k_j}   L_{k_j}^{(0)} \left( GW ( \mathbf q (c) ) \right).
\label{Lk1}
\end{align}
The first term corresponds to loops with an edge repeated four times in the
upper generations. As before, $l$ distinct vertices are chosen among the
children of the root (leading to the factor $\FM_l( \mathbf p (c))$). The
subtree issued from one of them contains a $1$-loop whereas the subtrees issued
from the $l-1$ other vertices contain $0$-loops. The choice of the vertex
followed by the $1$-loop induces the additional factor $l$.

\begin{figure}[t!]
\begin{center}
\includegraphics[width=\textwidth]{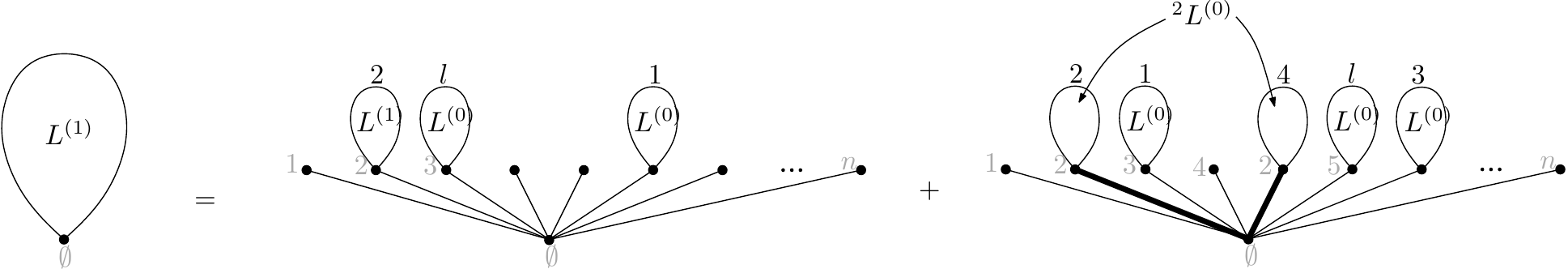}
\caption[1-loops]{\label{fig:1loops} Illustration of \eqref{Lk1} satisfied by
$L_{k}^{(1)} \left( UGW ( \mathbf p (c)) \right)$.
Numbers in gray next to vertices
count the children of the root, while numbers in black give the order of
appearance of first generation vertices in the loop. The two edges in fat are the same edge of the tree : it is the edge visited four times by the loop.
}
\end{center}
\end{figure}

The second term deals with loops where the edge repeated four times connects the
root to one of its children. Recall that a loop visits an edge connecting the
root twice before it can visit another edge connecting the root. Among the $l$
chosen edges connecting the root, $l-1$ will be repeated only twice (from the
root then towards it) and the last one will be repeated four times, giving a
total of $l+1$ visits of these $l$ edges both ways. There are now
$\binom{l+1}{2}$ choices for the ranks of the visits of the edge repeated four
times. Once these 2 ranks are fixed, and when the root has $n$ children, there
are $n (n-1) \cdots (n-l +1)$ choices for these edges giving the factor $\FM_l
(\mathbf p (c) ) l(l+1)/2$. The remaining of the loop then consists on the one
hand of $0$-loops lying in the subtrees issued from the $l-1$ first generation
vertices visited exactly twice and on the other hand of two $0$-loops lying in
the subtree issued from the first generation vertex visited four times, 
these two $0$-loops being disjoint (exception made of their starting point)
leading to the factor $\presuper{2}{L}_{k'}^{(0)} \left( GW ( \mathbf q (c)
\right)$. See Figure \ref{fig:1loops} for an illustration. We turn now to the
recursion relation satisfied by this term:
\begin{equation}
c^{-k'} \, \presuper{2}{L}_{k'}^{(0)} \left( GW ( \mathbf q (c)) \right) =
\sum_{l_1 + l_2 = 0}^{k'} \frac{\FM_{l_1 + l_2 + 1} ( \mathbf p (c) )}{c^{l_1
+l_2 + 1}}
 \sum_{\substack{k_1 + \cdots + k_{l_1 + l_2} \\ = k' -l_1 - l_2}} \prod_{j=1}^{l_1 + l_2} c^{-k_j}   L_{k_j}^{(0)} \left( GW ( \mathbf q (c))  \right).
 \label{2Lk0}
\end{equation}
In the above equation $l_1$ and $l_2$ represent the respective number of first
generation vertices visited by both loops. If the root has $n$ children, there
are $n (n-1) \cdots (n-l_1 + 1)$ possible choices for these vertices and their
order of appearance in the first loop. There are then $(n-l_1) \cdots (n-l_1 -
l_2 +1)$ choices for the vertices visited by the second loop.  This leads to the
term $\FM_{l_1 + l_2} (\mathbf q (c) ) = \FM_{l_1 + l_2 + 1} (\mathbf p (c)
)/c$. 

Finally, we need the following recursion equation for $L_{k}^{(1)} \left( GW (
\mathbf q (c) \right)$ obtained in a similar way than \eqref{Lk1}, the only
change residing in the factorial moments which are now related to $\mathbf q$
instead of $\mathbf p$ resulting in a shift from $l$ to $l+1$ in these terms:
\begin{align}
& c^{-k} L_{k}^{(1)} \left( GW ( \mathbf q (c)) \right) \notag \\
& = \sum_{l=1}^{k-2} \frac{\FM_{l+1}(\mathbf p (c))}{c^{l+1}}
l \sum_{k' = 2}^{k-l} c^{-k'} L_{k'}^{(1)} \left( GW ( \mathbf q (c) ) \right)
\sum_{\substack{k_1 + \cdots + k_{l-1} \\ = k - k' - l}} \prod_{j=1}^{l-1} c^{-k_j}   L_{k_j}^{(0)} \left( GW ( \mathbf q (c) ) \right) \notag \\
&  + \sum_{l=1}^{k-1} \frac{\FM_{l+1}(\mathbf p (c))}{c^{l+2}}
\frac{l(l+1)}{2} \sum_{k' = 0}^{k-l-1} c^{-k'} \presuper{2}{L}_{k'}^{(0)}
\left( GW ( \mathbf q (c)) \right) \sum_{\substack{k_1 + \cdots + k_{l-1}\\ = k - k' - l -1} } \prod_{j=1}^{l-1} c^{-k_j}   L_{k_j}^{(0)} \left( GW ( \mathbf q (c) ) \right).
\label{Lk1GW}
\end{align}

Recall that our aim is to compute the asymptotic expansion
\begin{equation}
c^{-k} L_k(UGW ( \mathbf p (c) ) ) =: a_k + \frac 1 {c^{1 \wedge \alpha}} b_k +o\left( \frac 1 {c^{1 \wedge \alpha}} \right).
\label{DL1}
\end{equation}
For that, we will need to compute the asymptotic expansions
\[
c^{-k} L^{(0)}_k(UGW ( \mathbf p (c) ) ) =: a^{(0)}_k + \frac 1 {c^{ \alpha}} b^{(0)}_k +o\left( \frac 1 {c^{ \alpha}} \right)
\]
and
\[
c^{-k} L^{(1)}_k(UGW ( \mathbf p (c) ) ) =: \frac 1 c b^{(1)}_k +o\left( \frac 1 c \right).
\]
Indeed, one can see that $c^{-k} L^{(1)}_k(UGW ( \mathbf p (c) ) )$ is of order $1/c$ from equation \eqref{Lk1}: the second term in the right hand side of \eqref{Lk1} is of order $1/c$ because of the factor $\FM_l (\mathbf p (c))/c^{l+1}$ and the first term is a finite sum of terms of the sequence $c^{-k'} L^{(1)}_{k'}(GW ( \mathbf q (c) ) )$. In turn, one can prove that these terms are of order $1/c$ by induction from equation \eqref{Lk1GW} due again to the presence of the term $\FM_{l+1} (\mathbf p (c))/c^{l+2}$.
The same line of reasoning allows to prove that loops with more repetitions than
$1$-loops will have a contribution of order $1/c^2$ because of a factor $\FM_l
(\mathbf p (c))/c^{l+2}$ in their recursion relation.

\bigskip
 
Identifying the main terms in \eqref{Lk0}, we get that the sequence
$\left(a_k^{(0)} \right)_{k \geq 0}$ satisfies the recursion relation of Catalan
numbers:
\[
a_k^{(0)} = \sum_{l = 1}^k \sum_{k_1 + \cdots + k_l = k-l} a_{k_1}^{(0)} \ldots a_{k_l}^{(0)} 
\]
for $k \geq 1$ and $a_0^{(0)} = 1$. Therefore, for every $k \geq 0$, one has $a_k^{(0)} = \CC(k)$.

Since equations \eqref{Lk0} and \eqref{Lk1} need equations \eqref{Lk0GW},\eqref{Lk1GW} and \eqref{2Lk0} to form a closed system of recursions, we introduce similarly the asymptotic expansions :
\begin{align*}
c^{-k} L^{(0)}_k(GW ( \mathbf q (c) ) ) & =: \CC(k) + \frac 1 {c^{\alpha}} \tilde{b}^{(0)}_k +o\left( \frac 1 {c^{\alpha}}\right); \\
c^{-k} L^{(1)}_k(GW ( \mathbf q (c) ) ) & =: \frac 1 c \tilde{b}^{(1)}_k +o\left( \frac 1 c \right); \\
c^{-k} \presuper{2}{L}^{(0)}_k(GW ( \mathbf q (c) ) ) & =: \presuper{2}{\tilde a} ^{(0)}_k + o\left( 1 \right).
\end{align*}

Now, let us compute the respective generating functions $B^{(0)}$,  $\tilde B^{(0)}$, $B^{(1)}$, $\tilde B^{(1)}$ and $\presuper{2}{\tilde A}^{(0)}$ of the numbers $b_k^{(0)}$, $\tilde b_k^{(0)}$, $b_k^{(1)}$, $\tilde b_k^{(1)}$ and $\presuper{2}{\tilde a}_k^{(0)}$. From equations \eqref{Lk0GW} and \eqref{DLcum}, we get
\begin{align*}
\tilde{b}^{(0)}_k & = \sum_{l = 1}^k f(l+1) \sum_{\substack{k_1 + \cdots + k_l \\ = k-l}} \CC(k_1) \ldots\CC(k_l) 
+ \sum_{l = 1}^k l \sum_{k' = 0}^{k-l} \tilde{b}^{(0)}_{k'} \sum_{\substack{k_1 + \cdots + k_{l-1} \\= k- k' - l}} \CC(k_1) \ldots\CC(k_{l-1}).
\end{align*}
This yields
\begin{align*}
\tilde B^{(0)}(x) & = \sum_{l \geq 1} f(l+1) x^l \left( T(x) \right)^l + \sum_{l \geq 1} l x^l \tilde B^{(0)}(x) \left( T(x) \right)^{l-1} \\
& = \frac{1}{x T(x)} \sum_{l \geq 2} f(l) \left( x T(x) \right)^l + \frac{x \tilde B^{(0)}(x)}{ \left( 1 - xT(x) \right)^2 }\\
& = \frac{F \left( xT(x) \right)}{x T(x)} + x \tilde B^{(0)}(x) \left(T(x) \right)^2
\end{align*}
where $F$ is the generating function $F(x) = \sum_{l \geq 2} f(l) x^l$. Hence
\[
\tilde B^{(0)}(x) = \frac{F \left( x T(x) \right)}{xT(x) \left( 1 - x T(x)^2\right)}.
\]
Similarly
\begin{align*}
B^{(0)}(x) & = \sum_{l \geq 1} f(l) x^l \left( T(x) \right)^l + \sum_{l \geq 1} l x^l \tilde B^{(0)}(x) \left( T(x) \right)^{l-1} \\
& = F \left( xT(x) \right)+ x \tilde B^{(0)}(x) T(x)^2
\end{align*}
leading to
\[
B^{(0)}(x) = \frac{2 F(x T(x))}{1 - x T(x)^2}.
\]

\bigskip

From equation \eqref{2Lk0}, we get
\[
\presuper{2}{\tilde a}_k^{(0)} = \sum_{l_1 + l_2 = 0}^k \sum_{k_1 + \cdots + k_{l_1 + l_2} = k-l_1 - l_2} \CC(k_1) \ldots\CC(k_{l_1 + l_2})
\]
leading to
\begin{align*}
\presuper{2}{\tilde A}^{(0)}(x) = \sum_{l_1 + l_2 \geq 0} x^{l_1 + l_2} T(x)^{l_1 + l_2} = \frac{1}{\left( 1 - x T(x) \right)^2} = T(x)^2.
\end{align*}
We now compute $\tilde B^{(1)}$ based on the following equation obtained from \eqref{Lk1GW}
\begin{align*}
\tilde b_k^{(1)} & = \sum_{l = 1}^{k-2} l \sum_{k' = 2}^{k-l} \tilde{b}^{(1)}_{k'} \sum_{\substack{k_1 + \cdots + k_{l-1} \\= k- k' - l}} \CC(k_1) \ldots\CC(k_{l-1})\\
& \qquad + \sum_{l = 1}^{k-1} \frac{l(l+1)}{2} \sum_{k' = 0}^{k-(l+1)} \presuper{2}{\tilde a}_{k'}^{(0)} \sum_{\substack{k_1 + \cdots + k_{l-1} \\ = k -k' -l - 1}} \CC(k_1) \ldots\CC(k_{l-1})
\end{align*}
leading to
\begin{align*}
\tilde B ^{(1)} (x) & = \sum_{l \geq 1} l x^l \tilde B ^{(1)} (x) T(x)^{l-1} + \sum_{l \geq 1} \frac{l (l+1)}{2} x^{l+1} \presuper{2}{\tilde A}^{(0)}(x) T(x)^{l-1} \\
& = x \tilde B ^{(1)} (x) T(x)^2 + \sum_{l \geq 1} \frac{l (l+1)}{2} \left( x T(x) \right)^{l+1}\\
& =  x \tilde B ^{(1)} (x) T(x)^2 + (x T(x))^2 \sum_{l \geq 0} \frac{l (l-1)}{2} \left( x T(x) \right)^{l-2}\\
& = x \tilde  B ^{(1)} (x) T(x)^2 + (x T(x))^2 \frac{1}{(1 - x T(x))^3} \\
& =  x \tilde B ^{(1)} (x) T(x)^2 + x^2 T(x)^5.
\end{align*}
Therefore
\[
\tilde B ^{(1)} (x) = \frac{x^2 T(x)^5}{1 - x T(x)^2}.
\]
Similarly, we compute $B^{(1)}$ from equation \eqref{Lk1}:
\begin{align*}
B ^{(1)} (x) & = \sum_{l \geq 1} l x^l \tilde B ^{(1)} (x) T(x)^{l-1} + \sum_{l \geq 1} \frac{l (l+1)}{2} x^{l+1} \presuper{2}{\tilde A}^{(0)}(x) T(x)^{l-1} \\
& =  x \tilde B ^{(1)} (x) T(x)^2 + x^2 T(x)^5 = \frac{x^2 T(x)^5}{1 - x T(x)^2}
\end{align*}
(note that $B^{(1)} = \tilde{B}^{(1)}$ !). 

This gives the theorem once we have identified $B^{(1)}$ with the moment generating function of the measure $\sigma^{\{1 \}}$.
Recall the following properties of the Hilbert transform $\HH$ of the semi-circle law:
\[
\HH(z) = \frac1z T \left( \frac{1}{z^2} \right) = \frac12 \left( z - \sqrt{z^2 -4} \right)
\]
leading to
\[
1 - \HH(z)^2 = \sqrt{z^2 -4} \times \HH(z).
\]
This gives
\begin{align*}
\HH^{(1)}(z): = \frac{1}{z} B^{(1)} \left( \frac{1}{z^2} \right) = \frac{\HH(z)^5}{1 - \HH(z)^2} = \frac{\HH(z)^4}{\sqrt{z^2 - 4}}
\end{align*}
which is the Hilbert transform of $\sigma^{\{ 1\} }$ computed in \eqref{H1}.
\end{proof}

\subsection{Applications to \ER and regular random graphs}

For \ER random graphs, $f(k) = 0$ for every $k$, therefore $F = 0$ and Theorem
\ref{UGW} directly gives Theorem \ref{ER}.

\bigskip

For a $c$-regular random graph, which converges locally towards a $c$ regular random tree, namely $UGW(\delta_c)$:
\begin{equation*}
\frac{\FM_k(\delta_c)}{c^k}  = \frac{c (c-1) \cdots (c-k +1)}{c^k} = 1 - \frac{k(k-1)}{2c} + o \left( \frac 1 c \right)
\end{equation*}
so that $f(k) = -k(k-1)/2$ for all $l \geq 0$ and
\[
F(x) = - \frac{x^2}{(1 - x)^3}.
\]
Therefore
\[
F(x T(x)) = - x^2 T(x)^5
\]
and $ \sigma_f^{\{1\}}  = - 2  \sigma^{\{1\}} $ and the perturbation of order
$1/c$ is the exact opposite as for \ER random graphs.

Note that this last density can be obtained from the Kesten Mc Kay formula
\cite{K59,mcKay81}:
\[
\mu \left( UGW(\delta_c) \right) (dx) = \frac{c}{2 \pi} \frac{\sqrt{4(c-1) - x^2 }}{c^2 - x^2} \mathbf{1}_{\{|x| < 2 \sqrt{c-1}\}}(x) \, dx.
\]

\bigskip

More generally, one can consider the family $\mathbf{p} (c)$ corresponding to
the laws of random variables of the form $c + c^{1- \frac \alpha 2} Y^{(c)}$
such that $E \left[ Y^{(c)} \right] = 0$, $E \left[( Y^{(c)} )^2 \right] \to 
\beta$ as $c \to \infty$ and all other moments of the $Y^{(c)}$'s stay bounded
with $c$. This example contains regular graphs (with $Y^{(c)} = 0)$ and \ER
random graphs (with $\alpha = \beta =1$ and $Y^{(c)}$ converging to a Gaussian
random variable as $c \to \infty$).  In this setting,
\[
\frac{\FM_k(\mathbf p (c))}{c^k}  = 1 + \beta \frac{k(k-1)}{2}\frac{1}{c^{\alpha}} -  \frac{k(k-1)}{2}\frac{1}{c} + o \left( \frac 1 {c^{1 \wedge \alpha}} \right)
\]
and the asymptotic expansion is given by
\[
m_k(\mu(\mathbf p (c)) ) = m_k (\sigma) + \frac 1 {c^{\alpha}} m_k \left( 2 \beta \sigma^{\{1\}} \right) - \frac 1 c m_k \left( \sigma^{\{1\}}  \right)  + o \left( \frac 1 {c^{1 \wedge \alpha}} \right).
\]

\section{Higher orders and edge of the spectrum}
\label{sec:order2}

\begin{proposition}
\label{proporder2}
The moments of the limiting spectral measure $\mu^c$ of the \ER random graph
have the following asymtotic expansion in $c$:
\[
m_k(\mu^c) = m_k \left(  \sigma +  \frac 1 c \sigma^{\{1\}} \right) + \frac{1}{c^2} d_k + o \left( \frac{1}{c^2} \right)
\]
where the generating series of  the numbers $d_k$ is given by
\[
D(x) =  \frac{x^3T(x)^7}{(1 - xT(x)^2)^3} \left(1 + 6 xT(x)^2 - 10 x^2 T(x)^4 + 4 x^3 T(x)^6 \right).
\]
\end{proposition}

\begin{proof}
Les us denote by $\mathcal P (c)$ the Poisson law with parameter $c$. We can write
\begin{equation}
m_k(\mu^c) = c^{-k} L_k(UGW ( \mathcal P (c) ) ) = c^{-k} L_k(GW ( \mathcal P (c) ) ) = a_k + \frac 1 c b_k + \frac{1}{c^2} d_k + o \left( \frac{1}{c^2} \right).
\label{DL2}
\end{equation}
The numbers $a_k$ and $b_k$ were computed in Section \ref{sec:ER}. The loops contributing to $d_k$ are the loops with two repetitions defined below.

\begin{definition}
Given a rooted tree,  we denote
\begin{itemize}
\item  by $2$-loops the loops started at the root and visiting each edge of the
tree either twice (a first time from the root and a second time towards the
root) or not at all with the exception of one edge visited six times
\item  by $(1,1)$-loops the loops started at the root and visiting each edge of
the tree either twice (a first time from the root and a second time towards the
root) or not at all with the exception of two distinct edges visited four times
each.
\end{itemize}
Furthermore, if $T$ is a random tree, we denote by $L_{k}^{(2)} (T)$ the expected number of $2$-loops in $T$ of length $2k$ and $L_{k}^{(1,1)} (T)$ the expected number of $(1,1)$-loops in $T$ of length $2k$.
\end{definition}

Let us first focus on $2$-loops. They satisfy the following recursion relation:
\begin{align}
& c^{-k} L_{k}^{(2)} \left( GW ( \mathcal P (c)) \right) \notag \\
& = \sum_{l=1}^{k-3} \frac{\FM_l(\mathcal P  (c))}{c^l} \,
l \sum_{k' = 3}^{k-l} c^{-k'} L_{k'}^{(2)} \left( GW ( \mathcal P (c) ) \right)
\sum_{\substack{k_1 + \cdots + k_{l-1}\\ = k - k' - l}} \prod_{j=1}^{l-1} c^{-k_j}  
L_{k_j}^{(0)} \left( GW ( \mathcal P (c) ) \right) \notag \\
& + \sum_{l=1}^{k-1} \frac{\FM_l(\mathcal P (c))}{c^{l+1}}
\binom{l+2}{3} \sum_{k' = 0}^{k-(l+2)} c^{-k'} \, \presuper{3}{L}_{k'}^{(0)}
\left( GW ( \mathcal P (c)) \right)  \sum_{\substack{k_1 + \cdots + k_{l-1}\\ = k - k' - l - 2} } \prod_{j=1}^{l-1} c^{-k_j}   L_{k_j}^{(0)} \left( GW ( \mathcal P  (c) ) \right)
\label{Lk2}
\end{align}
where $\presuper{3}{L}_{k'}^{(0)}(T)$ denotes the expectation of the number of
\emph{disjoint} and ordered triplets of $0$-loops started of total length $2k$ in a random tree $T$.
The justification is pretty similar as for relation \eqref{Lk1}.

We introduce the following notations for the asymptotic expansions :
\begin{align*}
c^{-k}  L_{k}^{(2)} \left( GW ( \mathcal P (c)) \right)
& = \frac{1}{c^2} d_k^{(2)} + o \left( \frac{1}{c^2} \right),\\
 c^{-k} \, \presuper{3}{L}_{k}^{(0)} \left( GW ( \mathcal P (c)) \right) & = \presuper{3}{ a}_{k}^{(0)} + o(1),
\end{align*}
and denote by $D^{(2)}$ and $ \presuper{3}{ A}^{(0)}$ the respective generating series of the numbers $d_k^{(2)}$ and $\presuper{3}{ a}_{k}^{(0)}$.

The expression of  $c^{-k} \, \presuper{3}{L}_{k}^{(0)} \left( GW ( \mathcal P (c)) \right)$ in terms of numbers  $c^{-k}{L}_{k}^{(0)} \left( GW ( \mathcal P (c)) \right)$ is similar to equation \eqref{2Lk0} and induces $ \presuper{3}{A}^{(0)}(x) = T(x)^3$. Equation \eqref{Lk2} gives
\begin{align*}
D^{(2)} (x) & = D^{(2)} (x) x T(x)^2 + \sum_{l \geq 1} \binom{l+2}{3} x^{l+2}
T(x)^3 T(x)^{l-1} \\
& = D^{(2)}(x) x T(x)^2+ (xT(x))^3 T(x)^4.
\end{align*}
Therefore
\[
D^{(2)} (x)  = \frac{x^3T(x)^7}{1 - x T(x)^2}.
\]

\bigskip

We now have to deal with $(1,1)$-loops. They satisfy the following recursion relation
\begin{align}
& c^{-k}  L_{k}^{(1,1)} \left( GW ( \mathcal P (c)) \right) \notag \\
& = \sum_{l=1}^{k-4} \frac{\FM_l(\mathcal P (c))}{c^l} \,
l \sum_{k' = 4}^{k-l} c^{-k'} L_{k'}^{(1,1)} \left( GW ( \mathcal P (c) ) \right)
\sum_{\substack{k_1 + \cdots + k_{l-1}\\= k - k' - l}} \prod_{j=1}^{l-1} c^{-k_j}  
L_{k_j}^{(0)} \left( GW ( \mathcal P (c) ) \right) \notag \displaybreak[0] \\
& + \sum_{l=1}^{k-4} \frac{\FM_l(\mathcal P (c))}{c^l} \,
\binom l 2 \sum_{k' = 4}^{k-l} c^{-k'} L_{k'}^{(1)} \left( GW ( \mathcal P (c) ) \right)
\sum_{k'' = 4}^{k-k'-l} c^{-k''} L_{k'}^{(1)} \left( GW ( \mathcal P (c) ) \right) \notag \\
& \qquad \qquad \qquad  \times 
\sum_{\substack{k_1 + \cdots + k_{l-2}\\= k - k' - k'' - l}} \prod_{j=1}^{l-2} c^{-k_j}  
L_{k_j}^{(0)} \left( GW ( \mathcal P (c) ) \right) \notag \displaybreak[0] \\
&  + \sum_{l=1}^{k-3} \frac{\FM_l(\mathcal P (c))}{c^{l+1}}
\binom{l+1}{2} \sum_{k' = 2}^{k-(l+1)} c^{-k'} \, \presuper{2}{L}_{k'}^{(1)}
\left( GW ( \mathcal P (c)) \right) \notag \\
& \qquad \qquad \qquad \times  \sum_{\substack{k_1 + \cdots + k_{l-1} \\ = k - k' - l - 1}} \prod_{j=1}^{l-1} c^{-k_j}   L_{k_j}^{(0)} \left( GW ( \mathcal P (c) ) \right) \notag \displaybreak[0] \\
& + \sum_{l = 1}^{k-2} \frac{\FM_l(\mathcal P (c))}{c^{l+1}}
\binom{l+1}{2} \sum_{k' = 0}^{k - l - 1} c^{-k'} \, \presuper{2}{L}_{k'}^{(0)}
\left( GW ( \mathcal P (c)) \right) \notag \\
& \qquad \qquad \qquad 
\times (l - 1 ) \sum_{k'' = 2}^{k - l - 1 - k '} c^{-k''} \, {L}_{k''}^{(1)} \left( GW ( \mathcal P (c)) \right)
\sum_{\substack{k_1 + \cdots + k_{l-2} \\ = k - k' - k'' - l - 1}} \prod_{j=1}^{l-2} c^{-k_j}   L_{k_j}^{(0)} \left( GW ( \mathcal P (c) ) \right) \notag\\
& + \sum_{l=2}^{k-4} \frac{\FM_l(\mathcal P (c))}{c^{l+2}}
\binom{l+2}{4} \binom 4 2 \sum_{k' = 0}^{k-(l+2)} c^{-k'} \, \presuper{2}{L}_{k'}^{(0)}
\left( GW ( \mathcal P (c)) \right) \notag \\
& \qquad \qquad \qquad 
\times \sum_{k'' = 0}^{k- k'-(l+2)} c^{-k''} \, \presuper{2}{L}_{k''}^{(0)} \left( GW ( \mathcal P (c)) \right)
\sum_{\substack{k_1 + \cdots + k_{l-2} \\ = k - k' - k'' -l - 2}} \prod_{j=1}^{l-2} c^{-k_j}   L_{k_j}^{(0)} \left( GW ( \mathcal P (c) ) \right) \displaybreak[0].
\label{Lk11}
\end{align}
where $\presuper{2}{L}_{k'}^{(1)}(T)$ denotes the expectation of the number of
\emph{disjoint} and ordered pairs loops of total length $2k$ in a random tree $T$, with one of the loops being a $1$-loop and the other being a $0$-loop.

The two first terms correspond to loops with repeated egdes only in the upper generations. Such loops visit $l$ distinct vertices among the children of the root (leading to the factor $\FM_l( \mathcal P (c))$). In the first term, the subtree issued from one of them contains a $(1,1)$-loop whereas the subtrees issued from the $l-1$ other vertices contain $0$-loops. In the second term, two of the subtrees issued from the $l$ vertices contain a $1$-loop whereas the subtrees issued from the $l-2$ other vertices contain $0$-loops.

The third and fourth terms deal with loops where one edge repeated four times connects the root to one of its children. Recall that a loop visits an edge connecting the root twice before it can visit another edge connecting the root. Among the $l$ chosen edges connecting the root, $l-1$ will be repeated only twice (from the root then towards it) and the last one will be repeated four times, giving a total of $l+1$ visits of these $l$ edges both ways. There are now $\binom{l+1}{2}$ choices for the ranks of the visits of the edge repeated four times. Once these 2 ranks are fixed, and when the root has $n$ children, there are $n (n-1) \cdots (n-l +1)$ choices for these edges and $l$ finally giving the factor $\FM_l (\mathcal P (c) ) l(l+1)/2$.
In the third term, the remaining of the loop consists on the one hand of $0$-loops lying in the subtrees issued from the $l-1$ first generation vertices visited exactly twice and on the other hand of one $1$-loop together with a disjoint $0$-loop lying in the subtree issued from the first generation vertex visited four times, 
leading to the factor $\presuper{2}{L}_{k'}^{(1)} \left( GW ( \mathcal P (c)
\right)$.  In the fourth term, the remaining of the loop consists on the one hand of a $1$-loop lying in one of the subtrees issued from the $l-1$ first generation vertices visited exactly twice, the rest of these subtrees being visited by $0$-loops, and on the other hand of a pair of disjoint $0$-loops lying in the subtree issued from the first generation vertex visited four times.

The last term deals with loops where the two edges repeated four times connect the root to one of its children.

We turn now to the recursion relation satisfied by $c^{-k} \, \presuper{2}{L}_{k}^{(1)} \left( GW ( \mathcal P (c)) \right)$:
\begin{align}
& c^{-k} \, \presuper{2}{L}_{k}^{(1)} \left( GW ( \mathcal P (c)) \right) \notag \\
& = 2 \sum_{l_1 = 1}^{k} \sum_{ l_2 = 0}^{k - l_1} \frac{\FM_{l_1 + l_2} ( \mathcal P (c) )}{c^{l_1
+l_2}} l_1 \sum_{k' = 1}^{k - l_1 - l_2}  c^{-k'} \, {L}_{k'}^{(1)} \left( GW ( \mathcal P (c)) \right) \notag \\
& \qquad \qquad \qquad 
\times
 \sum_{\substack{k_1 + \cdots + k_{l_1 + l_2 - 1} \\ = k -l_1 - l_2 - k'}}
\prod_{j=1}^{l_1 + l_2 - 1} c^{-k_j}   L_{k_j}^{(0)} \left( GW ( \mathcal P (c))
 \right) \notag \\
& + 
2 \sum_{l_1 = 1}^{k} \sum_{ l_2 = 0}^{k - l_1} \frac{\FM_{l_1 + l_2 } ( \mathcal P (c) )}{c^{l_1
+l_2 }}\binom{ l_1+1}{2} 
\sum_{k' = 0}^{k - l_1 - l_2 -1 }  c^{-k'} \, \presuper{2}{L}_{k'}^{(0)} \left( GW ( \mathcal P (c)) \right) \notag \\
& \qquad \qquad \qquad 
\times
 \sum_{\substack{k_1 + \cdots + k_{l_1 + l_2 - 1} \\ = k -l_1 - l_2 - k' - 1}}
\prod_{j=1}^{l_1 + l_2 - 1} c^{-k_j}   L_{k_j}^{(0)} \left( GW ( \mathcal P (c))
 \right). \label{2Lk1}
\end{align}
As in equation \eqref{2Lk0}, the parameters $l_1$ and $l_2$ represent the respective number of first generation vertices visited by both loops. If the root has $n$ children, there are $n (n-1) \cdots (n-l_1 + 1)$ possible choices for these vertices and their order of appearance in the first loop. There are then $(n-l_1) \cdots (n-l_1 - l_2 +1)$ choices for the vertices visited by the second loop.  This leads to the term $\FM_{l_1 + l_2} (\mathcal P (c) )$.

The first term on the right hand side of the equation verified by
$\presuper{2}{L}_{k}^{(1)}$ deals with pairs of loops where the repeated edge
lies in the upper generations of the tree. In this case, if the $1$-loop visits
$l_1$ first generation vertices, exactly one of the subtrees issued from these
vertices is visited by a $1$-loop (hence the multiplicative factor $l_1$), and
the $l_1 + l_2 - 1$ remaining subtrees are visited by $0$-loops. The factor $2$
comes from the fact that the pair of loops ($1$-loop and $0$-loop) is ordered.

Finally, the second term on the right hand side of the equation verified by
$\presuper{2}{L}_{k}^{(1)}$ deals with pairs of loops where the repeated edge
connects a first generation vertex to the root. Consider such a pair of loops.
As in equation \eqref{2Lk0},  if the $1$-loop visits $l_1$ vertices in the first
generation of a tree, there are $\binom{l+1}{2}$ choices for the ranks of the
visits of the edge repeated four times. The remaining of the loops then consists
on the one hand of $0$-loops lying in the subtrees issued from the $l_1 + l_2
-1$ first generation vertices visited exactly twice and on the other hand of two
$0$-loops lying in the subtree issued from the first generation vertex visited
four times, these two $0$-loops being disjoint (exception made of their starting
point). Here again, the factor $2$ comes from the fact that the pair of loops
($1$-loop and $0$-loop) is ordered.

\bigskip

Let us introduce the following asymptotic expansions:
\begin{align*}
c^{-k}  L_{k}^{(1,1)} \left( GW ( \mathcal P (c)) \right)
& = \frac{1}{c^2} d_k^{(1,1)} + o \left( \frac{1}{c^2} \right);\\
 c^{-k} \, \presuper{2}{L}_{k}^{(1)} \left( GW ( \mathcal P (c)) \right) & =
\presuper{2}{ b}_{k}^{(1)} + o(1)
\end{align*}
and denote by $\presuper{2}{ B}^{(1)}(x) $ the generating series of the numbers
$ \presuper{2}{ b}_{k}^{(1)}$. Using the fact that $\FM_{k} ( \mathcal P
(c) ) = c^k$ for every $k \geqslant 0$, the equation \eqref{2Lk1} yields
\begin{align*}
\presuper{2}{ B}^{(1)}(x) & = 2 B^{(1)} (x) \left( \sum_{l_1\geq 1} l_1 x^{l_1} T(x)^{l_1 -1} \right)\left( \sum_{ l_2\geq 0} (xT(x))^{l_2} \right)\\
& \qquad \qquad \qquad + 2 \presuper{2}{ \tilde A}^{(0)}(x)  \left( \sum_{l_1\geq 1} \frac{l_1 (l_1 +1)}{2} x^{l_1+1} T(x)^{l_1 -1} \right)\left( \sum_{l_2 \geq 0} (xT(x))^{l_2} \right)\\
& = \frac{2 x^3 T(x)^8}{1 - x T(x)^2} + 2 x^2 T(x)^6 =  \frac{2 x^2 T(x)^6}{1 - x T(x)^2}.
\end{align*}
In addition, If we denote by $D^{(1,1)} $ the generating series of the numbers $ d_k^{(1,1)}$, the equation \eqref{Lk11} yields
\begin{align*}
& D^{(1,1)}(x)\\
 &  = D^{(1,1)}(x) xT(x)^2 + \left( B^{(1)}(x) \right)^2 \sum_{l \geq 0} \frac{l (l-1)}{2} x^l T(x)^{l-2} \\
& \qquad  \qquad \qquad + \presuper{2}{B}^{(1)}(x) \sum_{l \geq 0} \frac{l (l+1)}{2} x^{l+1} T(x)^{l-1}
+ \left( \presuper{2}{\tilde A}^{(0)}(x) \right)^2 \sum_{l \geq 2} \binom{l+2}{4} \binom  4 2 x^{l+2} T(x)^{l-2}\\
& \qquad \qquad \qquad +  \presuper{2}{\tilde A}^{(0)}(x) {\tilde B}^{(1)}(x) \sum_{l \geq 1} \frac{(l+1)l(l-1)}{2} x^{l+1} T(x)^{l-1} \\
&  = D^{(1,1)}(x) xT(x)^2 + \left( \frac{x^2T(x)^5}{1 - x T(x)^2} \right)^2 x^2 T(x)^3 +  \frac{2 x^2 T(x)^6}{1 - x T(x)^2} x^2 T(x)^3 \\
& \qquad \qquad \qquad \qquad \qquad  + 6x^4 T(x)^9 + 3  \frac{ x^5 T(x)^{11}}{1 - x T(x)^2}\\
& = \frac{x^4T(x)^9}{(1 - xT(x)^2)^3} \left(8 - 11 xT(x)^2 +4 x^2 T(x)^4 \right).
\end{align*}

Finally, the generating series of the term of the second order of the moments of
\ER spectral measure is given by :
\begin{align*}
D(x) = D^{(2)}(x) + D^{(1,1)}(x) =  \frac{x^3T(x)^7}{(1 - xT(x)^2)^3} \left(1 + 6 xT(x)^2 - 10 x^2 T(x)^4 + 4 x^3 T(x)^6 \right).
\end{align*}
\end{proof}

In the spirit of Theorem \ref{ER}, we would like to interpret the numbers $d_k$
as the moments of a measure with null mass. To that aim, let us compute the
Stieljes transform of their generating series $D$:
\begin{align*}
\HH^{(2)} (z) & = \frac{1}{z} D \left( \frac{1}{z^2} \right) = \frac{\HH(z)^7}{1 - \HH(z)^2} \left(1+ 6 \HH(z)^2 -10 \HH(z)^4 +4 \HH(z)^6 \right) \\
&= \frac{\HH(z)^4 + 6 \HH(z)^6 - 10 \HH(z)^8 + 4 \HH(z)^{10}}{(z^2 -4)^{3/2}}.
\end{align*}
It is then easy to obtain
\begin{align*}
\lim_{\varepsilon \to 0} - \frac{1}{\pi} \HH^{(2)} (x + i \varepsilon) & = \frac{- 2 x^{10} + 25 x^{8} -113 x^6 + \frac{435}{2} x^4 - 155 x^2 + 19}{\pi (4- x^2)^{3/2}} \mathbf{1}_{|x|<2}
\end{align*}
which is not the density of a measure (this function has a non integrable singularity at $2$ and $-2$)!
This is due to the fact that the support of $\mu^c$ is not $[-2,2]$ but unbounded. 

It is possible to overcome this problem by trying to approximate the moments of $\mu^c$ by the moments of measures supported on intervals larger than $[-2, 2]$. Before stating our result, we introduce for $\alpha > 0$ the \emph{dilation operator} $\Lambda_{\alpha}$ that transforms a measure $\mu$ into the measure $\Lambda_{\alpha}(\mu)$ satisfying for every Borel set $A$, $\Lambda_{\alpha}(\mu) (A) = \mu (A/ \alpha)$.

\begin{theorem}
\label{edge}
The moments of $\mu^c$ satisfy the following asymptotic expansion:
\begin{align*}
m_k \left( \mu^c \right) & = m_k \left( \Lambda_{1 + \frac{1}{2c}} \left( \sigma + \frac 1 c \hat{\sigma}^{\{1\}} + \frac 1 {c^2} \hat{\sigma}^{\{2\}} \right) \right)  + o \left( \frac{1}{c^2} \right)
\end{align*}
where $\sigma$ is the semi-circle law, $\hat \sigma^{\{1\}}$ is the measure with null mass and density given by
\[
\hat f^{(1)} (x) = - \frac{x^4 - 5 x^2 + 4}{2 \pi \sqrt{4 - x^2}} \mathbf{1}_{|x|<2}
\]
and $\hat \sigma^{\{2\}}$ is the measure with null mass and density given by
\[
\hat f ^{(2)} (x) = - \frac{2x^8 -17 x^6 + 46 x^4 - \frac{325}{8} x^2 + \frac{21}{4}}{\pi \sqrt{4-x^2}}  \mathbf{1}_{|x|<2}.
\]
\end{theorem}

Before proving this theorem, let us make a brief comment. The measures appearing in the theorem are all supported on $]-2 - \frac{1}{c} ; 2 + \frac{1}{c} [$. This suggests that, in some sense, the right edge of the spectrum $\mu^c$ is located at $2 +\frac{1}{c}$. This can be compared with the spectrum of an infinite $d$-regular tree which is supported on $[-2\sqrt{d-1}, 2 \sqrt{d -1}]$ by the Kesten McKay formula; rescaling the spectrum by a factor $d^{-1/2}$ yields a support between $- 2 + 1/d$ and $2 - 1/d$ up to a correction of order $o(1/d)$.

\begin{proof}[Proof of Theorem \ref{edge}]
Fix $\alpha \in \RR$ and define
\begin{align*}
\hat m_{2k} & = \left(1 + \frac \alpha c \right)^{-2k} c^{-k} L_k \left( UGW \left( \mathbf{p} (c) \right) \right) \\
& = a_k + \frac 1 c \left( b_k - 2k \alpha a_k \right) + \frac{1}{c^2} \left(d_k - 2 k \alpha b_k + k (2k +1 ) \alpha^2 a_k \right) + o \left( \frac{1}{c^2} \right) \\
& =: a_k + \frac{1}{c} \hat b_k + \frac{1}{c^2} \hat d_k + o \left( \frac{1}{c^2} \right).
\end{align*}
If we can find $\alpha$ such that both $\hat b_k$ and $\hat d_k$ are the moments of two measures $\hat \mu^{(1)}$ and $\hat \mu^{(2)}$, then the following expansion holds:
\[
\int \left( \frac{x}{1 + \frac{\alpha}{c}}\right)^{2k} d \mu^c(x) = \int x^{2k} \sigma(x) dx + \frac{1}{c} \int x^{2k} \hat  \mu^{(1)}(dx) + \frac{1}{c^2} \int x^{2k } \mu^{(2)}(dx) + o \left( \frac{1}{c^2} \right)
\]
giving the asymptotic expansion announced in the theorem.

Now let us compute $\alpha$. To that aim, we need the generating series of $\hat d_k$:
\begin{align*}
\hat D(x) = D(x) -2 \alpha x B'(x) + 2 \alpha^2 x^2 T''(x) + 3 \alpha^2 x T'(x)
\end{align*}
With
\begin{align*}
T(x) &= \frac{1}{1-xT(x)}; \\
T'(x) &= \frac{T(x)^3}{1-xT(x)^2}; \\
T''(x) &= \frac{2T(x)^5(2-xT(x)^2}{\left( 1-xT(x)^2 \right)^3}; \\
B'(x) & = \frac{2 xT(x)^5 + 2x^2T(x)^7 -2 x^3 T(x)^9}{\left( 1-xT(x)^2 \right)^3}
\end{align*}
we obtain
\begin{align*}
\hat D(x)  & = \left( 1-xT(x)^2 \right)^{-3}  
\left( 4x^6T(x)^{13} -10 x^5 T(x)^{11} + \left( 6  + 4 \alpha \right) x^4 T(x)^9 \right. \\
& \qquad \qquad \qquad  \left. + \left( 1 -4 \alpha - \alpha^2 \right) x^3 T(x)^7
 + \left( -4 \alpha + 2 \alpha^2 \right) x^2 T(x)^5 + 3 \alpha^2 xT(x)^3 \right).
\end{align*}
We then have to compute the Stieljes transform of $\hat D$:
\begin{align*}
 \HH ^{\hat D} (z) & = \frac{1}{z} \hat D \left( \frac{1}{z^2} \right)   \\
& = \frac{ 4\HH^{10} -10 \HH^{8} + \left( 6  + 4 \alpha \right) \HH^6 + \left( 1 -4 \alpha - \alpha^2 \right) \HH^4
 + \left( -4 \alpha + 2 \alpha^2 \right)\HH^2 + 3 \alpha^2}{ \left( z^2 - 4  \right)^{3/2} }.
\end{align*}
The singularities of $\HH ^{\hat D}$ at $z = 2$ and $z =-2$ do not allow it to be the Stieljes transform of a measure except if the numerator is null for  $z = 2$ and $z =-2$. Since $\HH (2) = \HH (-2) = 1$, this gives the necessary condition
\[4 \alpha^2 - 4\alpha +1 = 0\]
with $\alpha = 1/2$ as the only solution. We then have, using the identity $1 -
\HH (z)^2 = \sqrt{z^2-4} \, \HH (z)$,
\[
\HH ^{\hat D} (z) = \frac{\HH^2 \left(16 \HH^6 - 8\HH^4 + 3 \right)}{4 \sqrt{z^2 - 4}}.
\]
This is the Stieljes transform of a measure with density given by
\[
\hat f ^{(2)} (x) = - \frac{2x^8 -17 x^6 + 46 x^4 - \frac{325}{8} x^2 + \frac{21}{4}}{\pi \sqrt{4-x^2}}  \mathbf{1}_{|x|<2}.
\]
In this setting, the perturbation of order $2$ is a measure with total mass $0$,
supported on $\left[-2 + \frac{1}{c} ; 2 + \frac{1}{c} \right]$ and with density
\[
\frac{1}{1 + \frac{1}{2c}} \hat f^{(2)} \left( \frac{x}{1 + \frac{1}{2c}}\right).
\]

Note that this also changes the perturbation of order $1$, indeed, recall that $\hat b_k = b_k - 2 \alpha k a_k = b_k - k a_k$. The generating series of $\hat b_k$ is
\[
\hat B (x) = B(x) - x T'(x) = \frac{x^2 T(x)^5 - x T(x)^3}{1 - xT(x)^2}.
\]
The corresponding Stieljes tranform is given by
\[
\HH ^{\hat B } (z) = \frac{\HH^4 - \HH^2}{\sqrt{z^2 - 4}}
\]
and corresponds to a measure with density
\[
\hat f^{(1)} (x) = - \frac{x^4 - 5 x^2 + 4}{2 \pi \sqrt{4 - x^2}} \mathbf{1}_{|x|<2}.
\]
Therefore, the perturbation of order $1$ is also a measure with total mass $0$,
supported on $\left[-2 + \frac{1}{c} ; 2 + \frac{1}{c} \right]$ and with density
\[
\frac{1}{1 + \frac{1}{2c}} \hat f^{(1)} \left( \frac{x}{1 + \frac{1}{2c}}\right).
\] \end{proof}

\section{Appendix: obstacles in the resolvent method}

In random matrix theory, the usual alternative to the moments method is the
so-called resolvent method. In this short section, we want to explain why this
method fails for our purpose. For the sake of simplicity, we focus on the
special case of the \ER random graph.

For a general probability measure $\mu$, the resolvent $R_{\mu}$ of $\mu$ is a function defined for every $z \in \mathbb C$ by
\[ 
R_{\mu} (z) = \int_{\RR} \frac{d\mu (x)}{x - z}. 
\]
The resolvent $Y_c$ of the spectral measure of an unimodular Galton Watson tree
with reproduction law Poisson with parameter $c$ satisfies the following
identity in law \cite{BL10}:
\[
Y_c(z) = - \frac{1}{z + \sum_{i=1}^{N(c)} Y_{c,i}(z) }
\]
where $N(c)$ is a Poisson random variable with parameter $c$ and the $Y_{c,i}$'s are iid copies of $Y_c$. The resolvent of $\mu^c$, the limiting spectral measure of the \ER random graph with parameter $c/n$ as $n \to \infty$ as defined at the begining of Section \ref{sec:ER}, is given by the expectation of $Y^{(c)} (z)  = \sqrt c Y_c (\sqrt c z)$.

In this setting, it is common to introduce
\[
f^{(c)} (u,z) = E \left[ e^{iuY^{(c)}(z)} \right]
\]
which satisfies the following functional equation (see \cite{KSV04} for \ER and \cite{BL10} Section 2. 2 for a more general case):
\begin{equation}
\label{eq:fbessel}
f^{(c)} (u,z)  = 1 - \sqrt u \displaystyle \int _0^{\infty} \displaystyle \frac{J_1 \left( 2 \sqrt{us} \right)}{\sqrt s} e^{i s z} e^{c \left( f^{(c)} (\frac s c  ,z) - 1\right)} ds
\end{equation}
where $J_1$ denotes the Bessel function of the first kind with index $1$.
We want to compute $E \left[ Y^{(c)} (z) \right] = - i \frac{\partial f^{(c)} }{\partial u} (0  ,z)$; to this aim, let us take the derivative of equation \eqref{eq:fbessel}:
\begin{align*}
\frac{\partial f^{(c)} }{\partial u} (u ,z) & = - \int_0^{\infty} \left(\frac{1}{2 \sqrt u} \frac{J_1(2 \sqrt{us})}{\sqrt s} + J'_1(2 \sqrt{us})  \right)e^{i s z} e^{c \left( f^{(c)} (\frac s c  ,z) - 1\right)} ds \\
& = - \int_0^{\infty} \left( - J_2(2 \sqrt{us}) +  2 \frac{J_1(2 \sqrt{us})}{\sqrt s}  \right)e^{i s z} e^{c \left( f^{(c)} (\frac s c  ,z) - 1\right)} ds.
\end{align*}
Taking $u \to 0$, this yields
\begin{equation}
\label{eq:EYc}
E \left[ Y^{(c)} (z) \right]  = i \int_0^{\infty} e^{i s z} e^{c \left( f^{(c)} (\frac s c  ,z) - 1\right)} ds.
\end{equation}
We want to compute an asymptotic expansion of $E \left[ Y^{(c)} (z) \right] $ as $c \to \infty$. Let us forget about the technical details and write the following \emph{formal} asymptotic expansion:
\[
f^{(c)} (u,z) = f_0 (u,z) + f_1(u,z) \frac 1 c + o\left( \frac 1 c \right).
\]
This also implies that
\[ E \left[ Y^{(c)} (z) \right]  = g_0 (z) + g_1(z) \frac 1 c + o\left( \frac
1 c \right) \]
with $g_i(z) = - i \frac{\partial f_i }{\partial u} (0  ,z) $.
Taking $c \to \infty$ in equation \eqref{eq:EYc} \emph{formally} gives
\begin{equation}
\label{g0func}
g_0(z) = i \int_0^{\infty} e^{isz} e^{is g_0(z)} ds = \frac{-1}{z + g_0(z)}.
\end{equation}
This is the functional equation satisfied by the Hilbert transform of the
semi-circle law, so we have recovered, at least at a \emph{formal} level, that
$\mu^{c} \to \sigma$ as $c \to \infty$. One can imagine that with some work,
this method can be rigorously justified.

However, if we pursue this method to compute the perturbation of order $1/c$, we will meet more serious problems. Still, we will continue the formal computations in order to try and recover the result of Theorem \ref{ER}. Indeed, equation \eqref{eq:fbessel} leads to
\begin{align}
f_0(u,z) + \frac 1 c f_1 (u,z)  & = 1 - \sqrt u \displaystyle \int _0^{\infty} \displaystyle \frac{J_1 \left( 2 \sqrt{us} \right)}{\sqrt s} e^{i s z} e^{c \left( f_0 (\frac s c  ,z) - 1\right)} e^{f_1 (\frac s c  ,z) }ds + o \left( \frac 1 c \right) \notag \\
& = 1 - \sqrt u \displaystyle \int _0^{\infty} \displaystyle \frac{J_1 \left( 2 \sqrt{us} \right)}{\sqrt s} e^{i s z} e^{is g_0(z) + \frac{s^2}{2c} \frac{\partial^2 f_0}{\partial u^2} (0,z) + \frac{s}{c} \frac{\partial f_1}{\partial u}(0,s) } ds + o \left( \frac 1 c \right). \label{formal exp}
\end{align}
This allows to compute $f_0$ :
\[
f_0(u,z) = 1 - \sqrt u \displaystyle \int _0^{\infty} \displaystyle \frac{J_1 \left( 2 \sqrt{us} \right)}{\sqrt s} e^{i s z} e^{is g_0(z)} ds = e^{-i \frac{u}{z + g_0(z)}} = e^{iu g_0(z)}.
\]
Replacing $f_0$ in \eqref{formal exp}, one gets
\begin{align*}
f_1(u,z) & =  - \sqrt u \displaystyle \int _0^{\infty} \displaystyle \frac{J_1 \left( 2 \sqrt{us} \right)}{\sqrt s}
e^{is (z + g_0(z))}
\left(s   \frac{\partial f_1}{\partial u}(0,z) + \frac{s^2}{2} \frac{\partial^2 f_0}{\partial u^2} (0,z)\right) ds \\
& =  - \sqrt u \displaystyle \int _0^{\infty} \displaystyle \frac{J_1 \left( 2 \sqrt{us} \right)}{\sqrt s}
e^{is (z + g_0(z))}
\left(s   \frac{\partial f_1}{\partial u}(0,z) - \frac{s^2}{2} \left( g_0(z) \right)^2 \right) ds 
\end{align*}
This last equation is where serious problems start: the integral is now 
divergent! Indeed, one has $\sqrt{us} . J_1(2 \sqrt{us}) \sim \sqrt{2 /
\pi} s^{1/4} u^{1/4} \cos (2\sqrt{us} - 3 \pi / 4)$ as $s \to \infty$.
Therefore, the expression of $f_1$ can only be considered at a formal level.
From there, we can compute the term of order $1/c$ of the Hilbert transform of
$\mu^{c}$:
\begin{align*}
\frac{\partial f_1}{\partial u}(0,z) & = - \int_0^{\infty} e^{is (z +
g_0(z))}\left(s   \frac{\partial f_1}{\partial u}(0,z) - \frac{s^2}{2} \left(
g_0(z) \right)^2 \right) ds\\
& = \frac{1}{i(z+g_0(z))} \int_{0}^{- i . (z + g_0(z)) \times \infty} e^{-t}
\left( \frac{it}{z + g_0(z)} \frac{\partial f_1}{\partial u}(0,z) +
\frac{t^2}{(z + g_0(z))^2}  \frac{(g_0(z))^2}{2}\right) dt
\end{align*}
where the last line is obtained with the change of variables $t = -i s (z +
g_0(z))$. Here again, taking no precautions and staying at a formal level
(writing $ \int_{0}^{- i . (z + g_0(z)) \times \infty} e^{-t} t^k dt = \Gamma (k
+1)$), one gets
\begin{align*}
\frac{\partial f_1}{\partial u}(0,z) &=  \frac{1}{(z+g_0(z))^2} \frac{\partial f_1}{\partial u}(0,z) - i \frac{(g_0(z))^2}{(z+g_0(z))^3} \\
& = (g_0(z))^2 \frac{\partial f_1}{\partial u}(0,z) + i (g_0(z))^5
\end{align*}
using \eqref{g0func} to obtain the last equality. This finally yields
\[
g_1(z) = - i \frac{\partial f_1}{\partial u}(0,z) = \frac{(g_0(z))^5}{1 - (g_0(z))^2}
\]
which is the Hilbert transform of $\sigma^{\{ 1\} }$ computed in \eqref{H1}.

\bigskip

Finally, let us mention that our best efforts to try to obtain the second order
term of Proposition \ref{proporder2} by an analogous formal computation failed.

\section{Appendix: numerical simulations}

We present here numerical simulations on 100 adjacency matrices of \ER graphs with 10000 vertices for $c=20$. Figure \ref{notrescaled} illustrates Theorem \ref{ER} and Figure \ref{rescaled} illustrates Theorem \ref{edge}. Concerning the plot of the second order perturbation in Figure \ref{rescaled}, one can notice a difference between the histogram and the density of the associated limiting measure. This difference can be explained by the fact that $n$ is not large enough even if it was large enough for the first order.

This raises the following interesting question: find a sequence of integers $n(c)$ (respectively $n_1(c)$, $n_2(c)$) depending on $c$ such that the moments of $\mu_{n(c)} ^c$ (respectively $c \left( \mu_{n_1(c)} ^c - \sigma \right)$,  
 $c^2 \left( \mu_{n_2(c)} ^c - \Lambda_{1 + \frac{1}{2c}} \left( \sigma + \frac 1 c \hat \sigma^{1} \right) \right)$) converge towards the moments of $\sigma$ (respectively $\sigma^{\{1\}}$, $\hat \sigma^{\{2\}}$).

\begin{figure}[h!]
\begin{center}
\includegraphics[width=\textwidth]{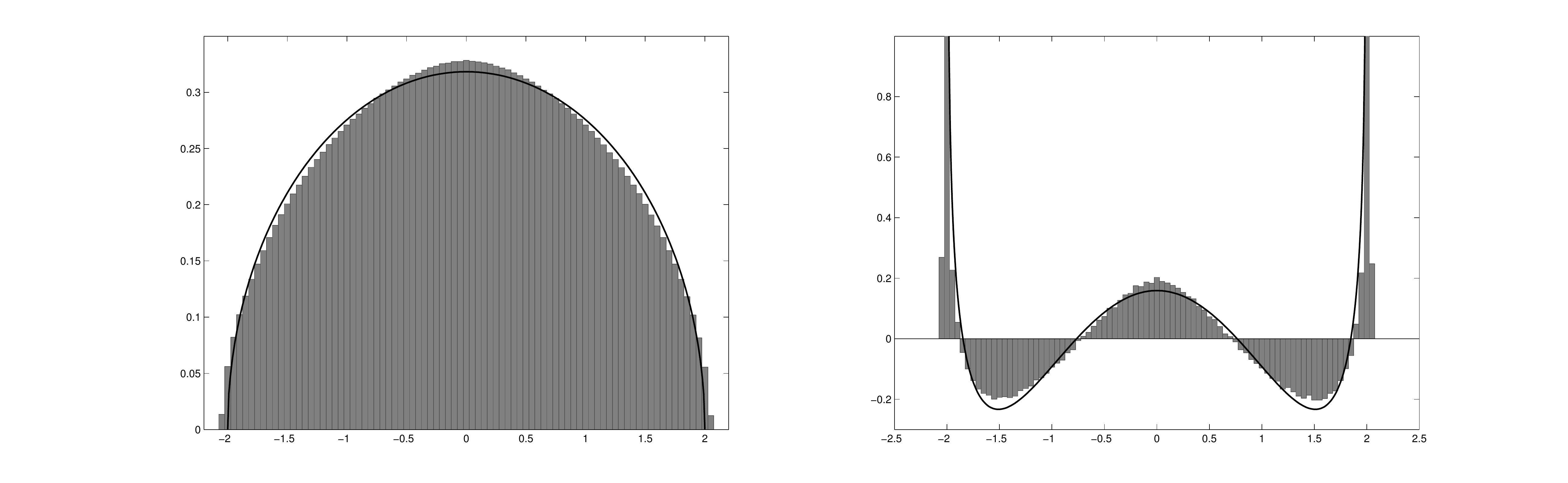}
\caption{\label{notrescaled} {\it Left:} Histrogram of $ \mu_n^{c} $ compared with the density of the semi-circle law  $\sigma$. {\it Right:} Histrogram of $c ( \mu_n^{c} - \sigma )$ compared with the density of $\sigma^{\{1\}}$.}
\end{center}
\end{figure}

\begin{figure}[h!]
\begin{center}
\includegraphics[width=\textwidth]{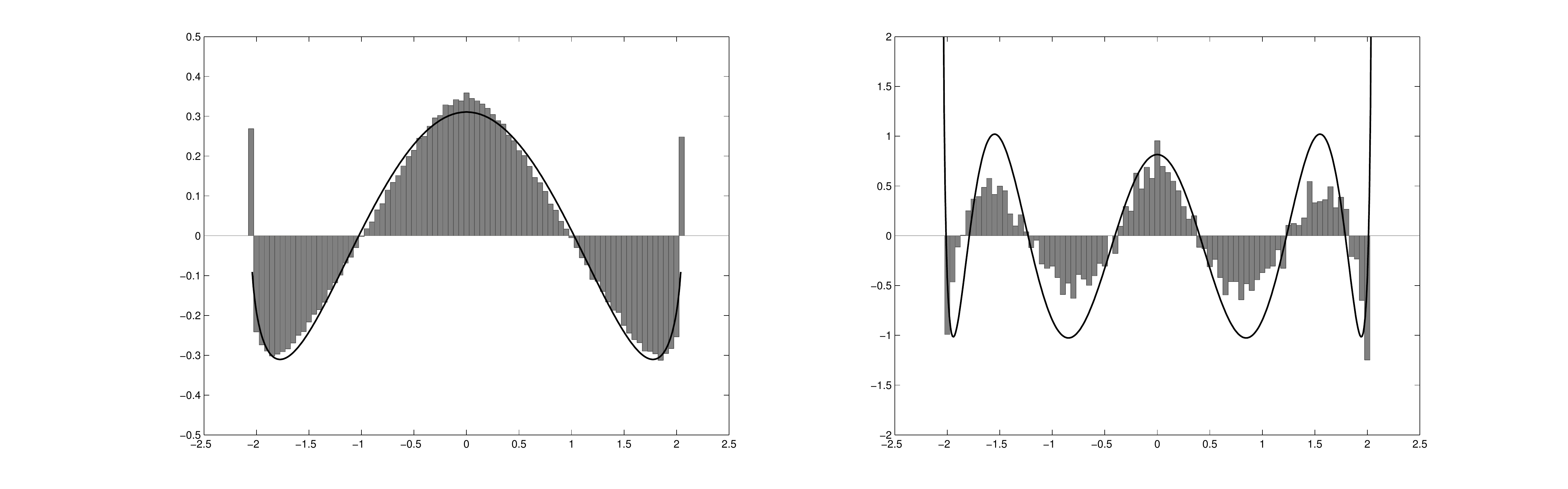}
\caption{\label{rescaled} {\it Left:} Histrogram of $c \left( \mu_n^{c} - \Lambda_{1 + \frac{1}{2c}} (\sigma) \right)$ compared with the density of $\Lambda_{1 + \frac{1}{2c}} \left(\hat \sigma^{\{1\}} \right)$. {\it Right:} Histrogram of $c^2 \left( \mu_n^{c} - \Lambda_{1 + \frac{1}{2c}} \left(\sigma + \frac{1}{c} \hat \sigma^{\{1\}} \right) \right)$ compared with the density of $\Lambda_{1 + \frac{1}{2c}} \left( \hat \sigma^{\{2\}} \right)$.}
\end{center}
\end{figure}

\newpage

\noindent {\bf Acknowledgements:} It is a pleasure for the authors to thank Gérard Ben Arous who initiated our interest in the subject and for fruitful discussions and insightful comments on the present work.

\addcontentsline{toc}{section}{References}
\bibliographystyle{abbrv}
\bibliography{spectre}

\bigskip

\noindent \textsc{Nathana\"el Enriquez} \verb|nenriquez@u-paris10.fr|, \\
\textsc{Laurent M\'enard} \verb|laurent.menard@normalesup.org|\\
Universit\'e Paris-Ouest Nanterre\\
Laboratoire Modal'X\\ 200 avenue de la R\'epublique\\
92000 Nanterre (France).

\end{document}